\documentclass[11pt]{article}
\usepackage{amssymb}
\usepackage{}
\usepackage{pifont}
\usepackage{amscd}
\usepackage{amsfonts}
\usepackage{amsmath, amsthm}
\usepackage{amssymb, amsxtra}
\usepackage[UTF8]{ctex}
\usepackage{graphicx}
\usepackage{tikz-cd}

\usepackage{multicol}
\usepackage{epstopdf}
\usepackage{epsf}
\usepackage{float}
\usepackage{extarrows}
\usepackage[a4paper, text={.8\paperwidth,.83\paperheight}, ratio=1:1]{geometry}
\usepackage[format=hang,font=small,textfont=it]{caption}
\usepackage{authblk}

\usepackage[colorlinks=true, linkcolor=blue, citecolor=blue]{hyperref}
\usepackage[english]{babel}
\usepackage{latexsym}

\usepackage{tikz}
\usetikzlibrary{calc}
\usepackage[tight]{subfigure}

\usepackage{color}
\definecolor{greenbean}{RGB}{199,237,204}

\newtheorem{thm}{Theorem}[]

\newtheorem{lemma}[thm]{Lemma}

\newtheorem{cor}[thm]{Corollary}
\newtheorem{eg}{Example}[]
\newtheorem{Qst}{Question}[]

\newtheorem{Def}[thm]{Definition}

\newtheorem{rmk}[thm]{Remark}

\def \ta{\tau}

\def \ta1{\tau_1}

\def \prodl{\prod\limits}

\begin{document}
\title{
 Topological and arithmetic characteristics about products of projective lines with complex tori
\footnotetext{\hspace{-1.8em}
E-mail address: J.-L. Mo: mojiali0722@126.com; M. Amram: meiravt@sce.ac.il; C. Gong: cgong@suda.edu.cn}}
\author[1]{Jia-Li Mo}
\author[2]{Meirav Amram}
\author[3]{Cheng Gong}

\affil[1]{\small{ College of Information and Network Engineering
 Anhui Science and Technology University, Fengyang, Anhui 233100, China}}
\affil[2]{\small{Department of Mathematics, Shamoon College of Engineering, Ashdod, Israel}}
\affil[3]{\small{Department of Mathematics, Soochow University, Suzhou, Jiangsu, China}}

\date{}

\maketitle

\abstract{In this paper,  we study non-planar degeneracies with cylindrical configurations. They could be constructed by the product $\mathbb{CP}^1 \times T$ of the projective plane and a complex torus with embedding $(m,n)$. We prove that their fundamental groups of Galois covers have an abelian subgroup of rank $m(2n-1)$ respectively, and the irregularity of these surfaces are at least $2mn-1$. Furthermore, we also use Chern numbers to compute the index of such surfaces and classify them.\\

{\bf Keywords} Chern number, complex torus, fundamental groups, Galois cover. \\

{\bf MR(2020) Subject Classification} 05E16, 14J10, 14J80,  20F55.

\section{Scientific background}\label{Introduction}

There has been extensive research about the topological and arithmetic properties of non-simply connected surfaces. In this paper, we investigate some of these properties of the surfaces.

An effective approach to studying the topological properties of algebraic surfaces is the application of degeneration techniques, particularly degeneration of the surface into a union of planes. The earliest application of this method was shown in the study of Hirzebruch surfaces in \cite{MRT}~and \cite{MoTe87}. As demonstrated in  ~\cite{AGRSV},~\cite{degree6},~ and~\cite{degree5}, this method has been successfully applied to surfaces of degrees~5,~6, and~8, revealing their topological properties.  Moreover, in references \cite{AGM1},~\cite{AGM2},~and \cite{AGM3}, we further investigated the topological properties of the Galois cover of Zappatic surfaces. Our results are related to Chisini's conjecture about algebraic surfaces and can be applied to the classification of algebraic surfaces. For details, please refer to \cite{Ku99} and \cite{Ku08}.

A promising direction for further work is the exploration of non-planar degenerations. Indeed, we have initiated preliminary investigations into such degenerations in~\cite{A}. This paper focuses on cylindrical non-planar degenerations, illustrated in Figure~\ref{cylinder}.

\begin{figure}[H]
\begin{center}
\scalebox{0.65}{\includegraphics{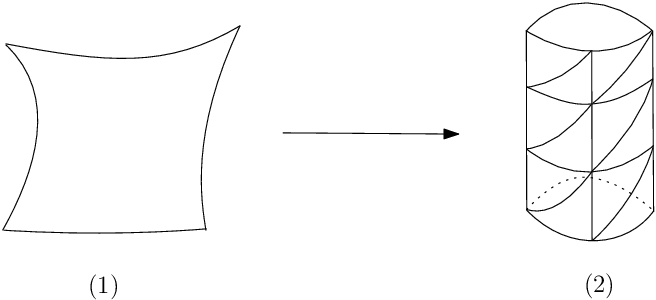}}
\end{center}
\setlength{\abovecaptionskip}{-0.15cm}
\caption{Cylindrical non-planar degeneration}\label{cylinder}
\end{figure}

Specifically, we degenerate the smooth surface (Figure \ref{cylinder}-(1)) into a cylindrical configuration: a union of multiple planes. In Figure \ref{cylinder}-(2), each plane is represented by a triangle.
The unfolding of the degenerate surface is shown in Figure \ref{small}.

\begin{figure}[H]
\begin{center}
\scalebox{0.95}{\includegraphics{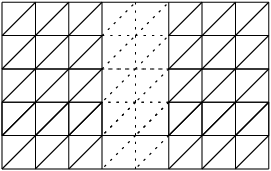}}
\end{center}
\setlength{\abovecaptionskip}{-0.15cm}
\caption{The degeneration of the  surface}\label{small}
\end{figure}

We use the symbol $X^{m,n}$ to represent a smooth surface that has degenerated to have a length of $n$ and a width of $m$ (i.e., comprising an $m\times n$ array of squares). For detailed information, please refer to Section~\ref{App}.

The Galois cover $X_{gal}$ of a surface $X$ serves as a geometric manifestation of Galois theory and reflects the topological properties of the surface. In \cite{Li}, the author elucidates the arithmetic properties of Galois covers and provides computational formulas for certain topological invariants. Furthermore, Galois covers can be applied to the study of moduli spaces of surfaces; for instance, in \cite{Ro}, the authors compute decompositions of these moduli spaces. Here, we aim to study the structure of fundamental group $\pi_1(X_{gal})$ of the Galois cover of a surface.

In this paper, as a first step towards computing $\pi_1(X^{m, n}_{gal})$ for surfaces $X^{m,n}$, here $m \geq 1, \ n \geq 2$, we provide a complete presentation of these fundamental groups for families $X^{2,4}$, $X^{2,n}$, and $X^{3,n}$. By applying techniques from \cite{RTV}, it follows directly that $\pi_1(X^{m, n}_{gal})$ is a subgroup of a certain quotient of Coxeter groups.
Furthermore, we employ Coxeter group theory to derive certain properties of $\pi_1(X^{m, n}_{gal})$, with a particular focus on its subgroups and quotient structures.
By analyzing the quotient group of $\pi_1(X^{m, n}_{gal})$, we are able to estimate the irregularity of the corresponding Galois cover in this setting.
The following theorem constitutes our main result.

\begin{thm}\label{main theorem} 
Group $\pi_1(X^{m, n}_{gal})$ contains a subgroup that is isomorphic to $\mathbb{Z}^{m(2n-1)}$.
\end{thm}
Since a simply connected space has trivial fundamental group, Theorem~\ref{main theorem} immediately implies that Galois covers $X^{m, n}_{gal}$ cannot be simply connected. Therefore, we obtain the following corollary.

\begin{cor}
The Galois covers of the surfaces $X^{m, n}$ are non-simply connected.
\end{cor}

Building on the findings presented in Theorem \ref{main theorem}, we can derive significant implications for the structure of the first homology group and the irregularity of the algebraic surface. This subgroup structure of the fundamental group indicates a rich topological complexity, facilitating a deeper understanding of the associated  properties of algebraic surface.

\begin{cor}
\begin{enumerate}
\item 
 The rank of  the first homology group $H_1(X^{m, n}_{gal},\mathbb{Z})$ is at least $2(2mn-1)$.
\item The irregularity of the algebraic surface $X^{m, n}_{gal}$ is at least $2mn-1$.
\end{enumerate}
\end{cor}

In algebraic geometry, when classifying a compact complex surface $X$, in addition to its own topological invariants (such as Betti numbers), the Chern numbers of its tangent bundle, $c^2_1(X)$ and $c_2(X)$, are two of the most fundamental numerical invariants. The positivity and negativity of Chern numbers is closely related to the Kodaira dimension of the surface.
In \cite{BHPV}, the authors discuss the geography problem of surfaces with given Chern numbers. The method of Galois covers can also be employed to provide many examples in the geography of surfaces.
We apply the formulas from \cite{MoTe87} to determine the Chern numbers of the Galois covers.

Another important topological invariant of algebraic surfaces $X$ is the index
of the intersection form on $H^2(X,\mathbb{Q})$, denoted by  $\tau(X)$. Hirzebruch's formula (see \cite{H}) says that
$$\tau (X)=\frac{1}{3}\left( c_1^2(X)-2c_2(X) \right).$$

By analyzing the index, we can draw connections between the topological characteristics of these surfaces and their geometrical  properties. We have the following theorem:

\begin{thm}\label{s}
\begin{enumerate}
   \item $\tau(X^{m,n}_{gal})=\frac{1}{3}(4m)!\cdot(8m-9)$, if $n=2$.  

   In particular, when $m=1$,  $X^{m,n}_{gal}$ are surfaces with a negative index.  
   
   When $m\geq 2$,  $X^{m,n}_{gal}$ are   surfaces with a positive index.  
   
   \item $\tau(X^{m,n}_{gal})=\frac{1}{3}(2mn)!\cdot (mn-3n+3)$, if $n\geq3$.

   In particular, when $m=1$, $X^{m,n}_{gal}$ are surfaces with a negative index. 

   When $m=2, n=3$,  $X^{m,n}_{gal}$ is an index 0 surface. Moreover, when $n>3$, such  surfaces have a negative index.

   When $m\geq 3$,  $X^{m,n}_{gal}$ are surfaces with a positive index. 
\end{enumerate}
\end{thm}

The index is said to be positive, zero, or negative according to whether the signature of the intersection form on $H^2(X^{m,n}_{gal},\mathbb{Q})$  is positive, zero, or negative. This classification is particularly helpful when constructing divisors with prescribed properties on the surface.

The paper is organized as follows. In Section \ref{App}, we give the constructions of the surfaces. In Section \ref{method}, we explain the methods we use and the terminology related to the paper. In Section \ref{CFG}, we use the Coxeter group theory to prove Theorem \ref{main theorem} and give some open questions. In Section \ref{bgzd}, we focus on the irregularity of the surfaces.
In Section \ref{Chern}, we compute the Chern numbers and topological indices of the surfaces.

\section{Construction of $X^{m,n}$}\label{App}

We take surface $X=\mathbb{CP}^1\times T$, where $T$ is an elliptic curve in $\mathbb{CP}^2$. 
We now describe the construction of surface  $X^{m,n}$. 
Note that $\mathbb{CP}^1$ can be embedded in $\mathbb{CP}^m~(m\geq2)$ and $T$ can be embedded in $\mathbb{CP}^{n-1}~(n \geq 3)$. Consequently, via Segre embedding $\mathbb{CP}^m \times \mathbb{CP}^{n-1} \hookrightarrow \mathbb{CP}^{m(n-1) + m + (n-1)}$, 
the surface $\mathbb{CP}^1\times T$ can be embedded into a larger projective space. We call such an embedding the $(m,n)$-embedding of $\mathbb{CP}^1\times T$ and denote the image of the embedding by $X^{m, n}$. 
 

To study fundamental group $\pi_1(X^{m,n}_{gal})$, we employ a degeneration approach following~\cite{AFT03}, combined with braid monodromy techniques~\cite{L, P}. This method allows us to degenerate the surface $X^{m,n}$ into a union of $2mn$ planes, as illustrated in Figure~\ref{small}.

We now explain the degeneration process in detail. Let $T$ be a normal elliptic curve in $\mathbb{CP}^{n-1}$. It can degenerate into an $n$-degree stick curve that forms a closed loop, as shown in Figure~\ref{nline}.

\begin{figure}[H]
\begin{center}
\scalebox{0.65}{\includegraphics{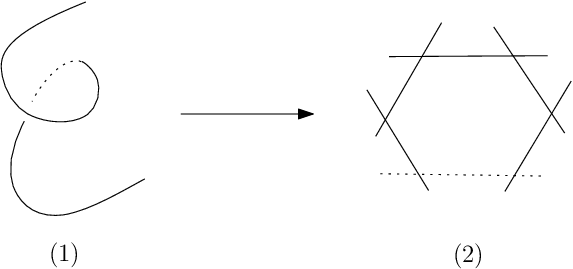}}
\end{center}
\setlength{\abovecaptionskip}{-0.15cm}
\caption{An elliptic normal curve}\label{nline}
\end{figure}

On the other hand, the rational normal curve $\mathbb{CP}^1$ in $\mathbb{CP}^{m}$ can degenerate into an $m$-degree stick curve that forms a path; see Figure~\ref{mlines}.

\begin{figure}[H]
\begin{center}
\scalebox{0.65}{\includegraphics{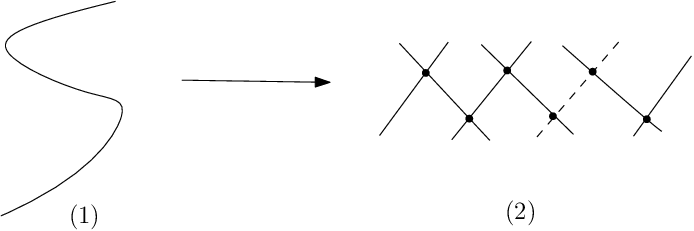}}
\end{center}
\setlength{\abovecaptionskip}{-0.15cm}
\caption{A rational normal curve}\label{mlines}
\end{figure}

The product surface $\mathbb{CP}^1 \times T$ consequently degenerates into $mn$ quadric surfaces. These are represented by the small squares in Figure~\ref{2small}.

\begin{figure}[H]
\begin{center}
\scalebox{0.95}{\includegraphics{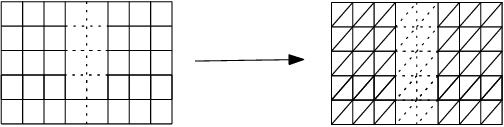}}
\end{center}
\setlength{\abovecaptionskip}{-0.15cm}
\caption{The degeneration of the surfaces $X^{m,n}$}\label{2small}
\end{figure}

Each quadric surface further degenerates into two planes, resulting in a total of $2mn$ planes. These are denoted by small triangles in Figure~\ref{2small}. This degenerate configuration corresponds to the surface $X^{m,n}$ introduced in Section \ref{Introduction}.

\section{Preliminaries and terminology}\label{method}
In this paper, we study the structure of the fundamental groups of the Galois covers of the surfaces $X^{m,n}$, building upon our previous studies of $X^{1,3}$ and $X^{1,n}$ in~\cite{AGTV02} and~\cite{AG04}, respectively. We denote the degenerations of \(X^{m,n}\) by \(X_0^{m,n}\), and these degenerate surfaces with as shown in Figure \ref{Figmn}  will be used throughout the paper for certain values of \(m\) and \(n\).

\begin{figure}[H]
\begin{center}
\scalebox{0.75}{\includegraphics{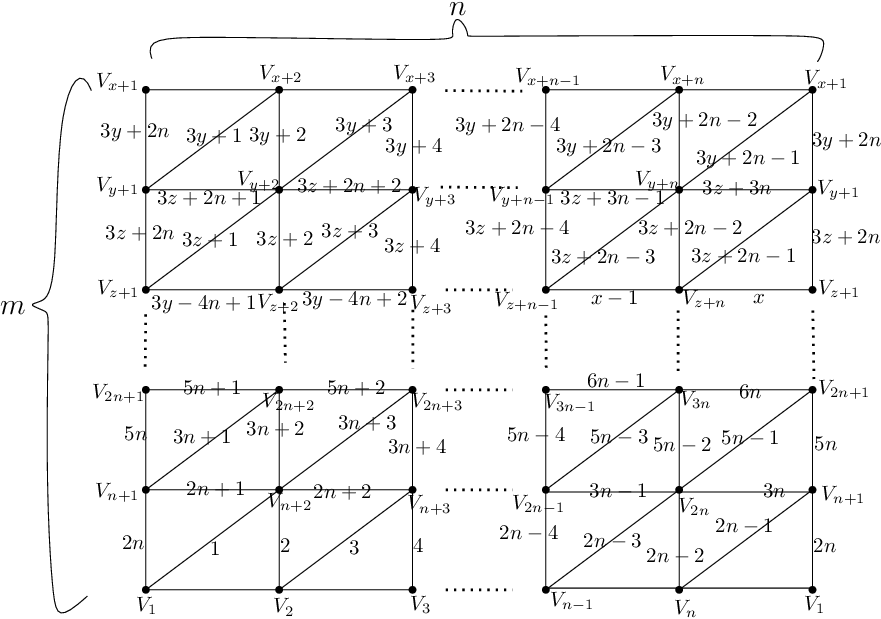}}
\end{center}
\setlength{\abovecaptionskip}{-0.15cm}
\caption{The degeneration of the surfaces $X^{m,n}$}\label{Figmn}
\end{figure}

To label edges in Figure \ref{Figmn}, we use $$x=mn; \ \ y=n(m-1); \ \ z=n(m-2).$$

In this section, we now begin to present all relevant information and lemmas concerning the general surfaces, whose notation will be $X^{m,n}$ for $m \geq 1, \ n \geq 2$, and whose degeneration appears above in Figure \ref{Figmn}.

We aim to determine the structure of fundamental group $\pi_1(X^{m,n}_{gal})$. First, we consider a generic projection from the degenerated surface $X^{m,n}_0$ to $\mathbb{CP}^2$, obtaining branch curve $C^{m,n}_0$, which is a union of $3y + 2n$ lines. These lines correspond to the images of the $3y + 2n$ edges in Figure~\ref{Figmn}, where the edges are numbered from left to right and from bottom to top, taking into account the identifications resulting from the degeneration of the torus. Through the regeneration process, we then obtain branch curve \( C^{m,n} \), which arises from a generic projection of the surface \( X^{m,n} \) to \( \mathbb{CP}^2 \). These processes can be intuitively understood from Figure~\ref{bbb}.

\begin{figure}[H]
\[
\begin{CD}
	X^{m,n}\subseteq \mathbb{CP}^N  @>\text{degeneration}>> X^{m,n}_0\subseteq \mathbb{CP}^N \\
	@V\text{generic~projection}VV                      @VV\text{generic~projection}V \\
	C^{m,n}\subset \mathbb{CP}^2 @<\text{regeneration}<< C^{m,n}_0\subset \mathbb{CP}^2\\
\end{CD}
\]
\caption{Degeneration and regeneration process.}
\label{bbb}
\end{figure}

Curve $C^{m,n}$ is a cuspidal curve of degree $2(3y + 2n)$ (since each line in $C^{m,n}_0$ regenerates into a conic or a pair of parallel lines in $C^{m,n}$). Applying  van Kampen Theorem for cuspidal curves as presented in~\cite{vk}, we compute group $\pi_1(\mathbb{CP}^2 - C^{m,n})$, the fundamental group of the complement of $C^{m,n}$ in $\mathbb{CP}^2$.

To derive an explicit presentation, we first adopt the global numbering of the lines in the degeneration as shown in Figure~\ref{Figmn}, and denote the standard generators by $1, 1', \dots, {3y+2n}, {(3y+2n)}'$.

We define the following group $$G^{m,n}=\frac{\pi_1(\mathbb{CP}^2-C^{m,n})}{\langle j^2,j'^2 \rangle},$$ 
where $\langle j^2, j'^2 \rangle$ denotes the normal subgroup generated by the elements $j^2$ and $j'^2$. 
There exists a surjection from $G^{m,n}$ onto the symmetric group $S_{2mn}$, 
defined by mapping each generator $j$ (respectively $j'$) to a transposition $(\alpha, \beta)$, 
where $j$ corresponds to number of an edge in Figure~\ref{Figmn},
and $\alpha, \beta$ are the labels of the two triangles intersecting along edge $j$. 
Furthermore, the fundamental group $\pi_1(X^{m,n}_{gal})$ of the Galois cover of $X^{m,n}$ 
is known to be the kernel of this surjection. 
For further details on Galois covers of such surfaces, we refer to~\cite{MoTe87a} and~\cite{MoTe87}.

Because it is very difficult to identify the group $G^{m,n}$, we define a quotient of $G^{m,n}$  as follows:
$$G^{m,n}_1 = \frac{\pi_1(\mathbb{CP}^2-C^{m,n})}{\langle j^{-1}j', \ \ j^2,j'^2 \rangle}=\frac{G^{m,n}}{\langle j^{-1}j' \rangle}.$$
This implies that we first determine a quotient of $\pi_1(X^{m,n}_{gal})$, which will subsequently enable us to deduce the structure of $\pi_1(X^{m,n}_{gal})$ itself.

Now we refer to the types of relations arising from the van Kampen Theorem in group $\pi_1(\mathbb{CP}^2 - C^{m,n})$ that we compute. The relations originating from each vertex depend on the local ordering of the lines at that point, as illustrated in Figure~\ref{Figmn}.
We then compile a list of relations according to the types of singularities:
\begin{enumerate}
\item For every branch point of a conic, a relation of the form $\gamma = \gamma'$ where $\gamma,\gamma'$ are certain conjugates of the generators ${j}$ and ${j'}$,
\item\label{vK2} for every node,
$[\gamma,\gamma'] = 1$ where $\gamma,\gamma'$ are certain conjugates of $i$ and $j$, where $i,j$ are the lines meeting in this node,
\item\label{vK3} for every cusp, $\langle\gamma,\gamma'\rangle=1$ where $\gamma,\gamma'$ are as in (\ref{vK2}),
\item\label{vK4} the ``projective relation'' $\prodl_{j=1}^{3y+2n} {j'}\ j = e$,
\item\label{par} an extra commutation relation for every pair of disjoint edges from Figure \ref{Figmn}, which meet in the branch curve.
\end{enumerate}

Group $\pi_1(\mathbb{CP}^2 - C^{m,n})$ is an amalgamated product of ``local groups", one for each vertex. These groups are generated by the elements corresponding to the lines meeting at the vertex, modulo the projective relation~(\ref{vK4}) and the relations of type~(\ref{par}). We refer to relations of type~(\ref{vK2}) and~(\ref{par}) as \emph{commutation relations}, and the relation of type~(\ref{vK3}) as a \emph{triple relation}. Once the presentation of $\pi_1(\mathbb{CP}^2 - C^{m,n})$ is obtained, we may proceed with the subsequent steps as outlined above.

To provide a clearer context for our work, we begin by explaining some aspects of the computation of fundamental groups.  
We first describe the types of singularities that occur in \(X_0^{m,n}\) (see Figure~\ref{Figmn}) and present related lemmas.

\begin{itemize}
\item In the degeneration of $X^{m,n}$, there exist points numbered $V_2, \dots, V_n$, each being the intersection of two lines, as illustrated in Figure~\ref{Fi1}-(1) for $i < j$.
\item There are also points numbered $V_1, V_{x+1}, \dots, V_{x+n}$, each likewise formed by the intersection of two lines, though with a distinct numbering scheme, as shown in Figure~\ref{Fi1}-(2) for $i < j$.
\end{itemize}

\begin{figure}[H]
\begin{center}
\scalebox{0.75}{\includegraphics{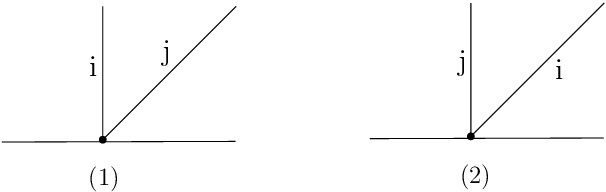}}
\end{center}
\setlength{\abovecaptionskip}{-0.15cm}
\caption{Intersections of two lines of two types}\label{Fi1}
\end{figure}

Now, we present the first lemma, which relates to the singularities illustrated in Figure~\ref{Fi1}-(1) and Figure~\ref{Fi1}-(2), where each singularity arises as the intersection of two lines.

\begin{lemma}\label{2pt-lemma}
Each vertex among \( V_2, \dots, V_n \) and \( V_1, V_{x+1}, \dots, V_{x+n} \) contributes a unique relation of type~(\ref{vK3}) to the presentation of the group
\[
G^{m,n}_1 = \frac{\pi_1(\mathbb{CP}^2 - C^{m,n})}{\left\langle j^{-1}j',\ j^2,\ j'^2 \right\rangle}.
\]
\end{lemma}

\begin{proof}
The relations in $\pi_1(\mathbb{CP}^2 - C^{m,n})$ arising from each of the vertices $V_2, \dots, V_n$ are:
\begin{equation*}\label{}
\langle{i},{j}\rangle=\langle{i'},{j}\rangle=\langle i^{-1}\ {i'}\ {i},{j}\rangle=e,
\end{equation*}
\begin{equation*}\label{}
{j'}={j}\ {i'}\ {i}\ {j}\ {i^{-1}}\ {i'^{-1}}\ {j^{-1}}.
\end{equation*}
The relations in $\pi_1(\mathbb{CP}^2 - C^{m,n})$ coming from each of the vertices $V_1, V_{x+1}, \dots, V_{x+n}$ are:
\begin{equation*}\label{}
\langle i',j\rangle=\langle i',j'\rangle=\langle i',j^{-1}\ j'\ j\rangle=e,
\end{equation*}
\begin{equation*}\label{}
i={j'}\ {j}\ {i'}\ {j^{-1}}\ {j'^{-1}}.
\end{equation*}

If we consider the constraints in group \(G^{m,n}_1\), the only remaining relation is \(\langle i, j \rangle = e\).
\end{proof}

We now turn to points that arise as intersections of six lines. In Figure~\ref{Figmn}, there are only two types of such points, which are as follows:

\begin{itemize}
\item These are the points numbered as 
$V_{n+2},\dots,V_{2n}, V_{2n+2},\dots,V_{3n},\dots, V_{z+2},\dots,V_{z+n},V_{y+2},\dots,V_{y+n}$, with $a<b<c<d<e<f$, see Figure \ref{Fi2}-(1).
\item Vertices $V_{n+1},V_{2n+1},\dots,V_{z+1},V_{y+1}$ are intersections of six lines as well, with $a<b<c<d<e<f$, but with a different labeling, see Figure \ref{Fi2}-(2).

\end{itemize}

\begin{figure}[H]
\begin{center}
\scalebox{0.75}{\includegraphics{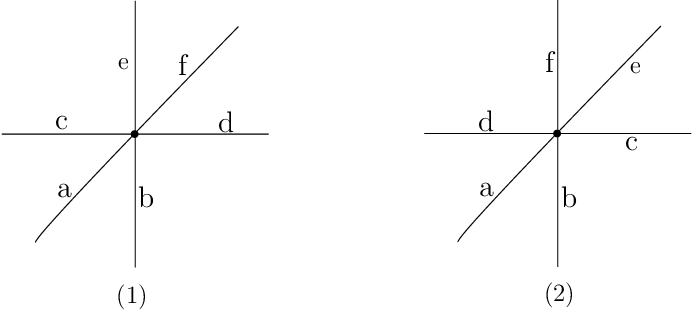}}
\end{center}
\setlength{\abovecaptionskip}{-0.15cm}
\caption{Two types of intersection of six lines}\label{Fi2}
\end{figure}

Continuously, we give some lemmas related to singularities in Figure \ref{Fi2}-(1) and \ref{Fi2}-(2) that are intersections of six lines.

\begin{lemma}\label{6pt-1-result}
The points that are intersections of six lines and of the type that appears in Figure \ref{Fi2}-(1), contribute the following relations to the presentation of $G^{m,n}_1:$
\begin{equation}\label{eq.1}
\langle a,b\rangle=\langle b,d\rangle=\langle d, f\rangle=\langle f, e\rangle=\langle e,c \rangle=\langle c,a \rangle=e,
\end{equation}
\begin{equation}\label{eq.2}
[a,d]=[a,e]=[a,f]=[b,c]=[b,e]=[b,f]=[d,e]=[d,c]=[c,f]=e,
\end{equation}
\begin{equation}\label{eq.3}
c\ a\ b\ a\ c = d\ f\ e\ f\ d.
\end{equation}
\end{lemma}

\begin{proof}
The relations arising from a point that is an intersection of six lines have already been presented in~\cite{ATXY}.  
Taking quotient \(G^{m,n}_1\) by the relations \(j = j'\) and \(j^2 = e\) for all generators, we immediately obtain:
\begin{equation}\label{}
\langle a,b\rangle=\langle a,c \rangle=\langle d, f\rangle=\langle e, f\rangle=\langle  a\ b\ {a},c \rangle=\langle f\ e\ f,d \rangle=e,
\end{equation}
\begin{equation}\label{}
[a,d]=[a,e]=[a,f]=[c,f]=e,
\end{equation}
\begin{equation}\label{}
[f, a\ b \ a]=[a\ b\ a, f\ e\ f]=e,
\end{equation}
\begin{equation}\label{cyclic1}
a\ b\ a\ c\ a\ b\ a= f\ e\ f\ d\ f\ e\ f.
\end{equation}
We then use equation \eqref{cyclic1} to further simplify these relations. Since some relations are redundant, we ultimately obtain \eqref{eq.1}--\eqref{eq.3} from the lemma.
\end{proof}

\begin{lemma}\label{6pt-2-result}
Each one of the points that is an intersection of six lines, as depicted in Figure~\ref{Fi2}-(2), contributes the following relations to \( G^{m,n}_1 \colon \) 

\begin{equation}\label{eq.4}
\langle a,b\rangle=\langle b,c\rangle=\langle c,e\rangle=\langle e, f\rangle=\langle f, d\rangle=\langle d, a\rangle=e
\end{equation}
\begin{equation}\label{eq.5}
[a, c]=[a, e]=[a, f]=[b,d]=[b,e]=[b,f]=[c,f]=[c,d]=[d,e]=e,
\end{equation}
\begin{equation}\label{eq.6}
c\ a\ b\ a\ c= d\ f\ e\ f\ d.
\end{equation}
\end{lemma}

\begin{proof}
The relations which are associated with this point, appear in \cite{ATXY}. The constraints in $G^{m,n}_1$ simplify them to \eqref{eq.4}-\eqref{eq.6}.
\end{proof}

\begin{lemma}\label{ii}
Given relations \eqref{eq.1}-\eqref{eq.3} in Lemma \ref{6pt-1-result}, the equalities  

\begin{equation}\label{44}
b~a~c~e~f=a~c~e~f~d=c~e~f~d~b=e~f~d~b~a=f~d~b~a~c=d~b~a~c~e 
\end{equation}
hold in $G^{m,n}_1$ as well.
\end{lemma}
    
\begin{proof}
We prove the first two equalities; the others follow similarly.

For the first equality, we proceed as follows: From relation~\eqref{eq.3}, we obtain
$
b = a~c~f~e~f~d~f~e~f~c~a.
$
Given that \([f, a] = [f, b] = [f, c] = e\), it follows that
\[
b = a~ c~ e~ f~ d~ f~ e~ c~ a,
\]
and therefore
$
b~ a~ c~ e~ f = a~ c~ e~ f~ d.$

For the second equality, we argue as follows: Again from relation~\eqref{eq.3}, we have
$
b~ c~ a~ c~ b~ e~ f = e~ f~ d.
$
Since \([b, e] = [b, f] = e\), this implies
\[
a~ c~ e~ f = c~ b~ e~ f~ d~ b.
\]
Furthermore, using \(\langle b, d \rangle = e\) and the identity \(c~ b~ e~ f~ d~ b = c~ e~ f~ b~ d~ b\), we derive
\[
a~ c~ e~ f = c~ e~ f~ d~ b~ d,
\]
and consequently,
$
a~ c~ e~ f~ d = c~ e~ f~ d~ b.
$
\end{proof}

Given relations \eqref{eq.4}--\eqref{eq.6} in Lemma~\ref{6pt-2-result}, we apply a similar method to prove the following equalities in \(G^{m,n}_1\):
\begin{equation}\label{444}
b~ a~ d~ f~ e = a~ d~ f~ e~ c = d~ f~ e~ c~ b = f~ e~ c~ b~ a = e~ c~ b~ a~ d = c~ b~ a~ d~ f~.
\end{equation}

Suppose we have obtained the full presentation of \(G^{m,n}_1\)(i.e., all relations stated in the lemmas above). We then construct dual graph \(T\) corresponding to the degeneration. To each numbered edge in the degeneration, we associate a line connecting every pair of triangles that intersect along that edge, see Figure~\ref{CPmncombine}. 

\begin{figure}[H]
\begin{center}
\scalebox{0.65}{\includegraphics{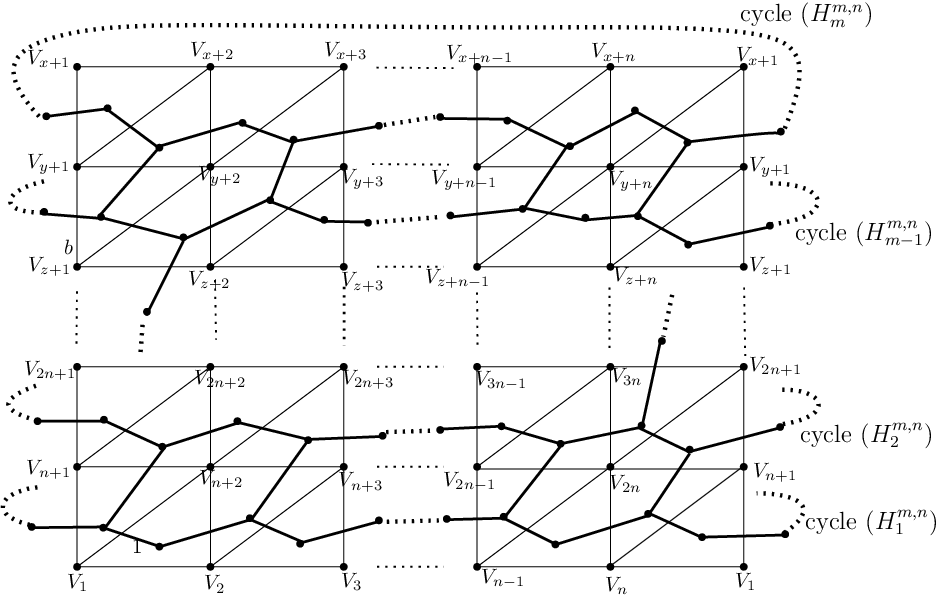}}
\end{center}
\setlength{\abovecaptionskip}{-0.15cm}
\caption{Dual graph in the  surface $X_0^{m,n}$}\label{CPmncombine}
\end{figure}

We note that the equalities from~\eqref{44} correspond to all cycles in the dual graph (see Figure~\ref{6circle}-(1)), that originate from intersections of six lines as shown in Figure~\ref{Fi2}-(1). Similarly, the equalities in~\eqref{444} correspond to all cycles in the dual graph (see Figure~\ref{6circle}-(2)), that arise from intersections of six lines in Figure~\ref{Fi2}-(2). For convenience, we refer to the relations defined by the equalities in~\eqref{44} and~\eqref{444} as the \emph{ \(E_6\) basic cycle relations}.

\begin{figure}[H]
\begin{center}
\scalebox{0.65}{\includegraphics{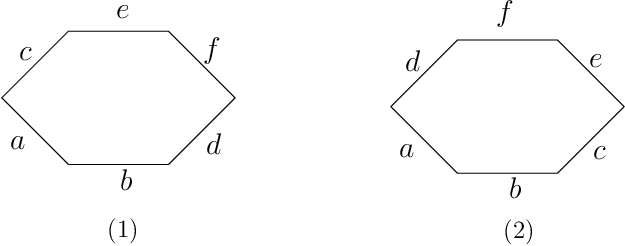}}
\end{center}
\setlength{\abovecaptionskip}{-0.15cm}
\caption{$E_6$ basic cycles}\label{6circle}
\end{figure}

Now we introduce the necessary settings and concepts to proceed with solving the problem. First, we recall group $A_{t,n}$ and then Coxeter group \( C(T) \).

\begin{Def}
{\bf Group $A_{t,n}$}. Let $t \geq 0, n \geq 1$. Fix a set $I$ of size $t$.

Group $A_{t,n}$ is generated by the $n^2|I|$ elements: $$x_{ij}~~(x \in X, i,j=1,2,\cdots,n)$$
with the defining relations

\begin{equation}
\begin{aligned}
&x_{ii}=1;\\
&x_{ij}~x_{jk}=x_{ik};\\
&x_{jk}~x_{ij}x_{ik} =1.
\end{aligned}
\end{equation}

and if $i, j, k, l$ are distinct,

\begin{equation}\label{68}
[x_{ij}, ~y_{kl}]=1.
\end{equation}
\end{Def}

For any $x,~y \in I$ ($x$ and $y$ can be equal). We will sometimes use $A_{t,n}$ to specify the set
of generators of $A$, this is the same group as $A_{t,n}$ for $t=|I|$.

The Coxeter group \( C(T) \) is associated with a graph \( T \) containing \( n \) vertices. This group is generated by the edges of \( T \), subject to the following relations:
\begin{itemize}
  \item \( u^2 = 1 \) for all edges \( u \in T \);
  \item \( u~v = v~u \) if \( u \) and \( v \) are disjoint edges;
  \item \( u~v~u = v~u~v \) if \( u \) and \( v \) intersect.
\end{itemize}

If the graph additionally contains triangular intersections as illustrated in Figure~\ref{3edge}, we define the quotient group \( C_Y(T) \) by imposing the following extra relation:
\begin{equation}\label{forkrel}
     [u, v~w~v]=1, \ \ u, v, w \in T
 \end{equation}
which meet in vertex $a$ (see Figure \ref{3edge}).
\begin{figure}[ht]
\begin{center}
\scalebox{0.65}{\includegraphics{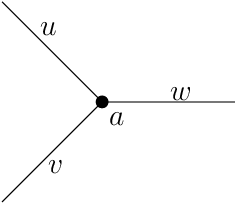}}
\end{center}\caption{Three edges meet in a vertex}\label{3edge}
\end{figure}

To better explore the connection between group $G^{m,n}_1$ and the quotient group $C_Y(T)$, we introduce the following additional relations, as also considered in \cite{RTV}.

Let $u_1, \ldots, u_r$ denote the edges of an arbitrary cycle (where $r$ is the length of the cycle). Each edge is relabeled as $u_i = (i-1, i)$ for $i=2, \ldots, r$, with $u_1 = (1, r)$; this labeling follows the convention illustrated in Figure~\ref{one}.

\begin{figure}[H]
\begin{center}
\scalebox{0.6}{\includegraphics{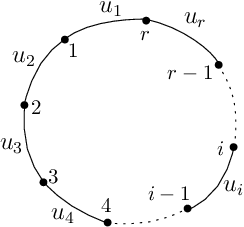}}
\end{center}\caption{Edge labeling convention for cycles}\label{one}
\end{figure}

Let $$\gamma_i = u_{i+2} \cdots u_r u_1 \cdots u_i, \text{ for } i=1, \ldots, r-2$$
and set $\gamma_{r-1} = u_1 \cdots u_{r-1}$ and $\gamma_r = u_2 \cdots u_r$. One can easily prove that
the equality $$\gamma_{r-1} = \gamma_{r}$$ holds in $S_n$.

We next outline the main idea and key steps of the proof. 

Taking the equality $\gamma_{r-1} = \gamma_r$ in $S_n$ as an additional relation within $C_Y(T)$, it can be viewed as imposing a constraint among generators for each $E_6$ basic cycle. Quotienting $C_Y(T)$ by the normal subgroup generated by all such $E_6$ basic cycle relations yields the isomorphism
\begin{equation}\label{formula}
 G^{m,n}_1 \ \cong \ C_Y(T)/\langle {E_6 \ \mbox{basic cycle relations}}\rangle,
\end{equation}
This isomorphism shows that $G^{m,n}_1$ retains the structure of $C_Y(T)$ while “flattening” all $E_6$ basic cycles. The elements of the form $\gamma_j\gamma_1^{-1}$ that are associated with non-$E_6$ basic cycles remain nontrivial in the quotient. Because $C_Y(T)$ is a quotient of $\pi_1(X^{m,n}_{gal})$, the lifting map pulls these elements back to $\pi_1(X^{m,n}_{gal})$.  These elements are then shown to be torsion‑free and pairwise commutative, thereby generating a free abelian subgroup $\mathbb{Z}^{m(2n-1)}$ inside $\pi_1(X^{m,n}_{gal})$. 

\section{Calculations of the Galois covers}\label{CFG}

\subsection{The Galois cover of a surface of $(2,4)$-type}\label{2,4}
In this subsection, we obtain results for  group $\pi_1(X^{2,4}_{gal})$. Figure \ref{Fig24} depicts the degeneration of the surfaces $X^{2,4}$.
\begin{figure}[H]
\begin{center}
\scalebox{0.75}{\includegraphics{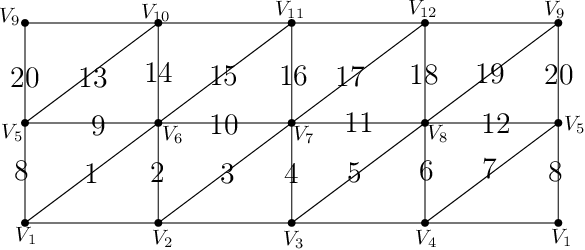}}
\end{center}
\setlength{\abovecaptionskip}{-0.15cm}
\caption{The degeneration of surface $X^{2,4}$}\label{Fig24}
\end{figure}

Using the lemmas established in Section \ref{method}, we compute the fundamental group $\pi_1(X^{2,4}_{gal})$. We obtain the following result:
\begin{thm}\label{24 theorem}
Group $\pi_1(X^{2,4}_{gal})$ contains a subgroup that is isomorphic to $\mathbb{Z}^{14}$.
\end{thm}

\begin{proof}
A great deal of information can be obtained directly from Figure \ref{Fig24}. Here, each vertex $V_i$ ($i=5,6,7,8$) is the intersection of six lines, while the remaining vertices are each the intersection of two lines.

Let $C^{2,4}$ be the branch curve of surface $X^{2,4}$. Fundamental group $\pi_1(\mathbb{CP}^2 - C^{2,4})$ of its complement in $\mathbb{CP}^2$ is generated by elements ${\{j,j'\}}^{20}_{j=1}$. Its presentation, derived from the van Kampen Theorem \cite{vk}, leads us to consider the quotient group $G^{2,4}_1$.

Relations in this group arise from the vertices $V_i$ ($i=6,7,8$). These relations can be obtained by applying Lemma \ref{6pt-1-result}:
\begin{eqnarray}\label{24.222}
& {\underline{Vertex~V_6}} \ \ \langle 1,2\rangle=\langle 2,10\rangle=\langle 10, 15\rangle=\langle 15, 14\rangle=\langle 14,9 \rangle=\langle 9,1 \rangle=e,\\
&[1,10]=[1,14]=[1,15]=[2,9]=[2,14]=[2,15]=[10,14]=[10,9]=[9,15]=e,
\end{eqnarray}
\begin{equation}\label{24.1}
 9\ 1\ 2\ 1\ 9 = 10\ 15\ 14\ 15\ 10. 
\end{equation}
\begin{eqnarray}
&{\underline{Vertex~V_7}}\ \ \langle 3,4\rangle=\langle 4,11\rangle=\langle 11, 17\rangle=\langle 17, 16\rangle=\langle 16,10 \rangle=\langle 10,3 \rangle=e,\\
&[3,11]=[3,16]=[3,17]=[4,10]=[4,16]=[4,17]=[11,16]=[11,10]=[10,17]=e,
\end{eqnarray}
\begin{equation}\label{24.2}
10\ 3\ 4\ 3\ 10 = 11\ 17\ 16\ 17\ 11.
\end{equation}
\begin{eqnarray}
&{\underline{Vertex~V_8}} \ \ \langle 5,6\rangle=\langle 6,12\rangle=\langle 12, 19\rangle=\langle 19, 18\rangle=\langle 18,11 \rangle=\langle 11,5\rangle=e,\\
&[5,12]=[5,18]=[5,19]=[6,11]=[6,18]=[6,19]=[12,18]=[12,11]=[11,19]=e,
\end{eqnarray}
\begin{equation}\label{24.4}
 11\ 5\ 6\ 5\ 11 = 12\ 19\ 18\ 19\ 12. 
\end{equation}

The relations in $G^{2,4}_1$ arising from  vertex $V_5$ are given by Lemma \ref{6pt-2-result}:
\begin{eqnarray}
& {\underline{Vertex~V_5}} \ \ \langle 7,8\rangle=\langle 8,9\rangle=\langle 9,13\rangle=\langle 13, 20\rangle=\langle 20, 12\rangle=\langle 12, 7\rangle=e,\\
&[7, 9]=[7, 13]=[7, 20]=[8,12]=[8,13]=[8,20]=[9,20]=[9,12]=[12,13]=e,
\end{eqnarray}
\begin{equation}\label{24.3}
7\ 8\ 9\ 8\ 7= 12\ 20\ 13\ 20\ 12.
\end{equation}

Each of the other eight vertices is an intersection of two lines and yields a corresponding relation of type (\ref{vK3}).
\begin{equation} \langle 1,8\rangle=\langle 1,2\rangle=\langle 2, 3\rangle=\langle 3,4\rangle=\langle 4,5\rangle=\langle 5,6\rangle=\langle 6,7\rangle=\langle 7,8 \rangle=e,
\end{equation}
\begin{equation}\label{66}
\langle 13,20\rangle=\langle 13,14\rangle=\langle 14, 15\rangle=\langle 15,16\rangle=\langle 16,17\rangle=\langle 17,18\rangle=\langle 18,19\rangle=\langle 19,20 \rangle=e.
\end{equation}

Additional commutation relations of the form (\ref{par}) exist for lines constituting $C^{2,4}$ that intersect; for example, we have $[1, 3] = e$. The projective relation (\ref{vK4}) also holds in $G^{2,4}_1$, but it is redundant.

Dual graph $T$ consists of 20 lines, numbered $1, 2, \dots, 20$, which serve as generators for the group $C(T)$.
So far, we have identified all relations of types (\ref{vK2}) and (\ref{vK3}) that hold in $C(T)$.

From relation (\ref{24.3}) and the relation $[1,12] = e$, we derive $[1, 8~9~8] = e$. Similarly, from (\ref{24.1}) and $[13,10] = e$, we obtain $[13, 14~9~14] = e$.
By iterating this computational process throughout the graph, we obtain all relations of type (\ref{forkrel}). Together with the relations of types (\ref{vK2}) and (\ref{vK3}), this yields the complete presentation of the group $C_Y(T)$.

Furthermore, the presentation of $G^{2,4}_1$ includes the $E_6$ basic cycle relations, which are (\ref{24.1}), (\ref{24.2}), (\ref{24.4}), and (\ref{24.3}). Figure \ref{24dualgraph} depicts the dual graph of the degeneration $X_0^{2,4}$ and these $E_6$ basic cycles around $V_5, V_6, V_7,$ and $V_8$.
\begin{figure}[H]
\begin{center}
\scalebox{0.75}{\includegraphics{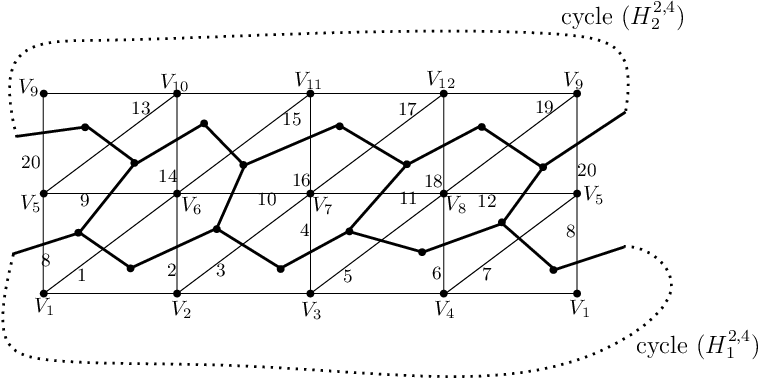}}
\end{center}
\setlength{\abovecaptionskip}{-0.15cm}
\caption{Dual graph of the degeneration $X_0^{2,4}$}\label{24dualgraph}
\end{figure}

The calculations above show that the defining relations for $G^{2,4}_1$ are precisely those of $C_Y(T)$, with the addition of the $E_6$ basic cycle relations. Therefore, the isomorphism formula in (\ref{formula}) holds here for $m=2, n=4$.

Furthermore, an analysis of the relations in $G^{2,4}_1$ confirms that while it includes all $E_6$ basic cycle relations, it does not contain relations corresponding to the two types of cycles shown in Figure \ref{24dualgraph}:
\begin{itemize}
    \item The cycle $(H^{2,4}_1)$ generated by generators $1, 2, 3, 4, 5, 6, 7, 8$ in the dual graph;
    \item The cycle $(H^{2,4}_2)$ generated by generators $13, 14, 15, 16, 17, 18, 19, 20$ in the dual graph.
\end{itemize}

According to preliminary knowledge, we have the following exact sequence:\\
$$
\begin{tikzcd}
0 \arrow[r] 
  & \pi_1(X^{2,4}_{gal}) \arrow[r] \arrow[d, "\theta"] 
  & G^{2,4} \arrow[r] \arrow[d] 
  & S_{16} \arrow[r] \arrow[d, equal] 
  & 0 \\
0 \arrow[r] 
  & K_1(T) \arrow[r] 
  & G^{2,4}_1 \arrow[r] 
  & S_{16} \arrow[r] 
  & 0
\end{tikzcd}
$$

According to Proposition 4.1 in \cite{RTV}, kernel $K_1(T)$ in the second exact sequence is generated as a subgroup by the elements $\gamma_j \gamma^{-1}_i$. Our analysis confirms that the elements corresponding to the two cycle types $(H^{2,4}_1)$ and $(H^{2,4}_2)$ are contained in $K_1(T)$.

We now define the specific elements for each cycle. For cycle $(H^{2,4}_1)$, we set:
\begin{align*}
\gamma_i &= (i+2) \cdots 8 \cdot 1 \cdots i \quad \text{for } i = 1, \ldots, 6,  &\gamma_7 = 1 \cdots 7, \
\gamma_8 &= 2 \cdots 8.
\end{align*}
Similarly, for cycle $(H^{2,4}_2)$, we define:
\begin{align*}
\beta_j &= (j+14) \cdots 20 \cdot 13 \cdots (j+12) \quad \text{for } j = 1, \ldots, 6,  &
\beta_7 = 13 \cdots 19, \
\beta_8 &= 14 \cdots 20.
\end{align*}

Thus, based on the above-mentioned proposition, we find that $K_1(T)$ contains the following 14 elements:
$$\gamma_{i}\gamma_{1}^{-1}, \beta_{j}\beta_{1}^{-1}, \text{ for } i=2, \ldots, 8; \ j=2, \ldots, 8.$$

Now, we prove that the above 14 elements commute with each other. Clearly, the 14 elements listed above are distinct，and not equal to the identity. Since map $\theta$ is surjective, we conclude that $\pi_1(X^{2,4}_{gal})$ also contains the inverse images of these 14 elements, which are indeed the 14 elements themselves. Using Lemmas \ref{6pt-1-result}, \ref{6pt-2-result}, and \ref{ii}, we show that the above 14 elements commute  with each other.

We take cycle $(H^{2,4}_1)$ as an example to illustrate this. Given $\gamma_1 = 3\cdot4\cdot5\cdot6\cdot7\cdot8\cdot1$, $\gamma_7 = 1\cdot2\cdot3\cdot4\cdot5\cdot6\cdot7$, and $\gamma_8 = 2\cdot3\cdot4\cdot5\cdot6\cdot7\cdot8$, we write the commutativity of $\gamma_7\gamma^{-1}_1$ and $\gamma_8\gamma^{-1}_1$ as follows
$$\gamma_7\gamma^{-1}_1\gamma_8 \gamma^{-1}_1= \gamma_8\gamma^{-1}_1\gamma_7\gamma^{-1}_1.$$
This equality becomes
$$\gamma_7\gamma^{-1}_1\gamma_8 = \gamma_8\gamma^{-1}_1\gamma_7.$$
Now we substitute the defined expressions for $\gamma^{-1}_1$, $\gamma_7$, and $\gamma_8$ into each side of the equality, and get
$$\gamma_7\gamma^{-1}_1\gamma_8 =( 1\cdot2\cdot3\cdot4\cdot5\cdot6\cdot7)(1\cdot8\cdot7\cdot6\cdot5\cdot4\cdot3)(2\cdot3\cdot4\cdot5\cdot6\cdot7\cdot8)$$ and $$\gamma_8\gamma^{-1}_1\gamma_7=(2\cdot3\cdot4\cdot5\cdot6\cdot7\cdot8)(1\cdot8\cdot7\cdot6\cdot5\cdot4\cdot3)(1\cdot2\cdot3\cdot4\cdot5\cdot6\cdot7).$$

We simplify both sides, using $\langle i, i+1 \rangle = e$ for $i=1,2,\dots,7$, $\langle 1, 8 \rangle = e$, and $[i, j] = e$ for $j \neq i, i\pm 1$ (with $i,j\in{1,2,\dots,8}$). The equality reduces to the relation $\langle 7, 8 \rangle = e$, which is already among the given relations. Hence, the equality holds.

Continuing the proof, we next show that none of the above 14 elements has finite order. When a finite-order element exists among the 14, the commutative group generated by them will have an extra relation. Because all of these elements are induced by the cycles $(H^{2,4}_1)$ and $(H^{2,4}_2)$ in graph $T$ and group $G^{2,4}_1$ is the quotient group of $A_{6,16}$, thus the above additional relation between elements induces an additional relation within $A_{6,16}$. However, $A_{6,16}$ admits no relations beyond the commutation relations. This contradiction completes the proof.



\begin{rmk}
The 14 elements we selected can be used to construct an abelian group by the method in this section. If other circles are chosen, however, the group constructed by a similar method is not necessarily abelian.
\end{rmk}

A detailed proof of the general case is given in Section~\ref{sec-mn}. We conclude that these elements generate a subgroup isomorphic to $\mathbb{Z}^{14}$, which completes the proof of the theorem.
\end{proof}

\subsection{The Galois covers of the surfaces of $(2,n)$-type and $(3,n)$-type}\label{sec-2n3n}

In this subsection, we study the degeneration of surfaces $X^{2,n}$ and $X^{3,n}$ for any $n$. See Figures~\ref{Fig2n} and~\ref{Fig3n}, respectively. In both degenerations, we have the same types of vertices (i.e., vertices that are intersections of either two lines or six lines) and the same rule for numbering the vertices and edges. This common structure enables us to deal with the general case.

To ensure clarity and avoid repetition, we present a brief proof for these two special cases. This illustrates how the general result arises naturally from specific examples. For further details, see Section~\ref{2,4} and the full proof of the general case.

We now present the conclusions from the degenerations of surfaces $X^{2,n}$ and $X^{3,n}$.

\begin{thm}\label{2n theorem}
$\pi_1(X^{2,n}_{gal})$ contains a subgroup that is isomorphic to $\mathbb{Z}^{4n-2}$.
\end{thm}

\begin{figure}[H]
\begin{center}
\scalebox{0.75}{\includegraphics{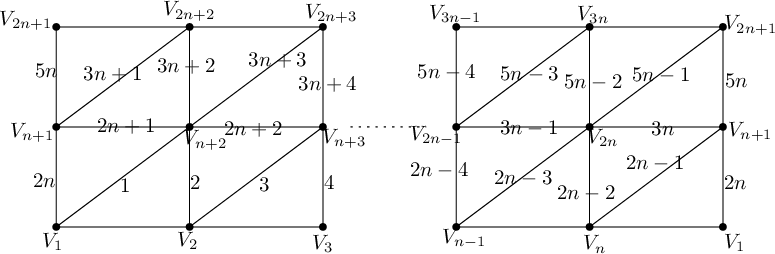}}
\end{center}
\setlength{\abovecaptionskip}{-0.15cm}
\caption{The degeneration of the surfaces  $X^{2,n}$}\label{Fig2n}
\end{figure}

\begin{proof}
Let $C^{2,n}$ denote the branch curve of surface $X^{2,n}$. The fundamental group $\pi_1(\mathbb{CP}^2 - C^{2,n})$ has a presentation with generators $\{j, j'\}_{j=1}^{5n}$; we consider this group $G_1^{2,n}$.

We take the elements $1, 2, \ldots, 5n$ as the generators. The isomorphism formula in (\ref{formula}) holds here for $m=2$, and $n$ is general number.

The dual graph in Figure~\ref{Fig2n} contains cycles of two types that are not incorporated into the group $G^{2,n}_1$. These cycles are:
\begin{itemize}
\item the cycle generated by the generators $1, 2, \dots, 2n$;
\item the cycle generated by the generators $3n+1, 3n+2, \dots, 5n$.
\end{itemize}
The techniques developed in Section~\ref{2,4} can be applied to these cycles as well.

Group $\pi_1(X^{2,n}_{gal})$ contains the following elements:
 $$\gamma_{i}{\gamma_{1}}^{-1},
\beta_{j}{\beta_{1}}^{-1}, \ \mbox{for}~{i=2, \dots, 2n; \ j=2, \dots, 2n}.$$

We prove that these $4n-2$ elements all have infinite order and pairwise commute. Consequently, they generate a subgroup isomorphic to $\mathbb{Z}^{4n-2}$, which completes the proof of the theorem.
\end{proof}

\begin{thm}\label{3n theorem}
$\pi_1(X^{3,n}_{gal})$ contains a subgroup that is isomorphic to ~$\mathbb{Z}^{6n-3}$.
\end{thm}

\begin{figure}[H]
\begin{center}
\scalebox{0.75}{\includegraphics{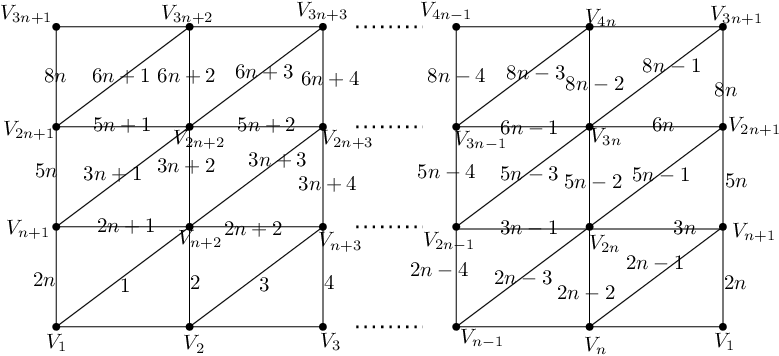}}
\end{center}
\setlength{\abovecaptionskip}{-0.15cm}
\caption{The degeneration of the surfaces $X^{3,n}$}\label{Fig3n}
\end{figure}

\begin{proof} 
Let $C^{3,n}$ denote the branch curve of surface $X^{3,n}$. The fundamental group of the complement $\pi_1(\mathbb{CP}^2 - C^{3,n})$, has a presentation with generators $\{j, j'\}_{j=1}^{8n}$; we denote this group by $G^{3,n}$. We take $1,2,\dots, 8n$ as generators. The isomorphism formula in (\ref{formula}) holds here for $m=3$, and $n$ is general number.

The dual graph (see Figure~\ref{CP3ncombine}) contains cycles of three types that are not incorporated into group $G^{3,n}_1$. These cycles are:
\begin{itemize}
\item The cycle $(H^{3,n}_1)$ generated by generators $1, 2, \dots, 2n$;
\item The cycle $(H^{3,n}_2)$ generated by generators $3n+1, 3n+2, \dots, 5n$;
\item The cycle $(H^{3,n}_3)$ generated by generators $6n+1, 6n+2, \dots, 8n$.
\end{itemize}

\begin{figure}[H]
\begin{center}
\scalebox{0.75}{\includegraphics{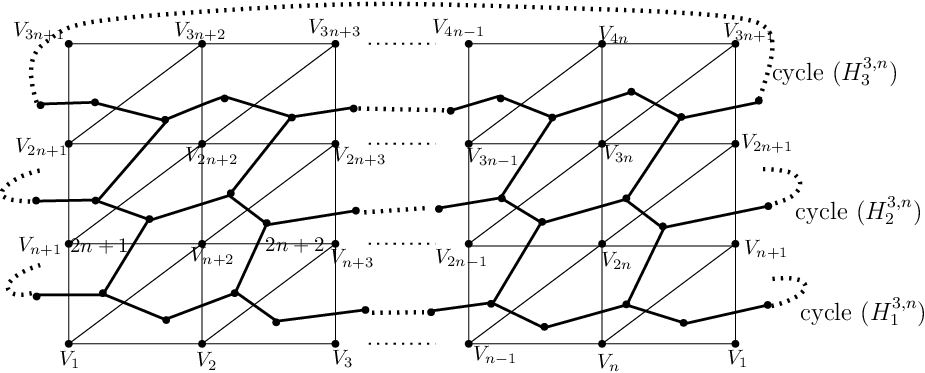}}
\end{center}
\setlength{\abovecaptionskip}{-0.15cm}
\caption{Dual graph in the  surface $X_0^{3,n}$}\label{CP3ncombine}
\end{figure}

 The techniques developed in Section~\ref{2,4} can be applied to these cycles as well.

Group $\pi_1(X^{3,n}_{gal})$ contains the following elements:
 $$\gamma_{i}{\gamma_{1}}^{-1},
\beta_{j}{\beta_{1}}^{-1}, \zeta_{k}{\zeta_{1}}^{-1}, \ \mbox{for}~{i=2, \dots, 2n; \ j=2, \dots, 2n; \ k=2, \dots, 2n}.$$

We prove that these $6n-3$ elements all have infinite order and pairwise commute. Consequently, they generate a subgroup isomorphic to $\mathbb{Z}^{6n-3}$, which completes the proof of the theorem.

\end{proof}

\subsection{The Galois covers of the surfaces of $(m,n)$-type}\label{sec-mn}

\begin{thm}\label{mn theorem}
Group $\pi_1(X^{m,n}_{gal})$ contains a subgroup that is isomorphic to ~$\mathbb{Z}^{m(2n-1)}
$.
\end{thm}

\begin{proof}
Let $C^{m,n}$ be the branch curve of the surface $X^{m,n}$. The fundamental group $\pi_1(\mathbb{CP}^2 - C^{m,n})$ has a presentation with generators ${\{j,j'\}}^{3mn-n}_{j=1}$; we denote this group by $G^{m,n}$.

Figure~\ref{CPmncombine} shows that the configuration has four types of vertices: upper edge vertices, lower edge vertices, inner vertices, and left edge vertices. The relations arising from these vertices, computed using Lemmas~\ref{2pt-lemma}, \ref{6pt-1-result}, and \ref{6pt-2-result}, yield all relations of types (\ref{vK2}) and (\ref{vK3}). Additionally, we obtain extra commutation relations of type (\ref{par}) for lines intersecting in $C^{m,n}$.

The remaining task is to derive the full set of relations of type (\ref{forkrel}). Based on the computations above, we now provide a detailed proof of this result.

Relations of type (\ref{forkrel}) can be readily identified on both the left and right sides of Figure~\ref{CPmncombine}. On the left side, the method for deriving (\ref{forkrel}) is similar for all cases; thus, it suffices to verify a single representative case as follows. Firstly, we give Figure~\ref{fork1} with a local piece from the left side of Figure~\ref{Figmn}. 
\begin{figure}[H]
\begin{center}
\scalebox{0.7}{\includegraphics{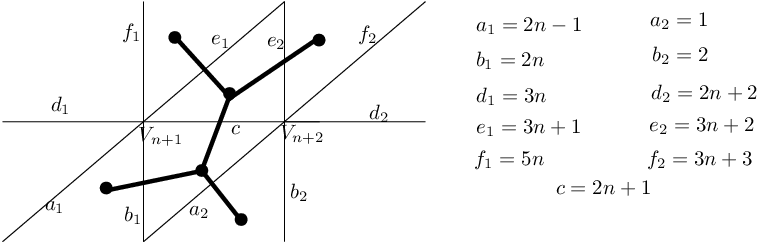}}
\end{center}
\setlength{\abovecaptionskip}{-0.15cm}
\caption{Local piece from Figure~\ref{Figmn} }\label{fork1}
\end{figure}
Secondly, let us summarize the relations that are related to vertex $V_{n+1}$, and that come from Lemma~\ref{6pt-2-result}:
\begin{eqnarray}\label{mn.1}
& \langle a_1,b_1\rangle=\langle b_1,c\rangle=\langle c,e_1\rangle=\langle e_1, f_1\rangle=\langle f_1, d_1\rangle=\langle d_1, a_1\rangle=e,\\
&[a_1, c]=[a_1, e_1]=[a_1, f_1]=[b_1,d_1]=[b_1,e_1]=[b_1,f_1]=[c,f_1]=[c,d_1]=[d_1,e_1]=e,\\
& c\ a_1\ b_1\ a_1 \ c=d_1\ f_1 \ e_1 \ f_1\ d_1.
\end{eqnarray}
And we summarize the relations that are related to vertex $V_{n+2}$ that come from Lemma~\ref{6pt-1-result}:
\begin{eqnarray}\label{mn.2}
& \langle a_2,b_2\rangle=\langle b_2,d_2\rangle=\langle d_2, f_2\rangle=\langle f_2, e_2\rangle=\langle e_2,c \rangle=\langle c,a_2 \rangle=e,\\
&[a_2,d_2]=[a_2,e_2]=[a_2,f_2]=[b_2,c]=[b_2,e_2]=[b_2,f_2]=[d_2,e_2]=[d_2,c]=[c,f_2]=e,\\
& c\ a_2\ b_2\ a_2\ c = d_2\ f_2\ e_2\ f_2\ d_2.
\end{eqnarray}
Next, we do the simplifications
\begin{eqnarray*}\label{mn.3}
&e=[a_2,d_1]=[a_2,f_1 \ e_1 \ f_1\ c\ a_1\ b_1\ a_1\ c\ f_1 \ e_1 \ f_1]\\
&\xlongequal{[a_2,f_1]=[a_2,e_1]=e}[a_2, c\ a_1\ b_1\ a_1\ c]
\xlongequal{[a_1,c]=[a_2,a_1]=e}[a_2,b_1\ c\ b_1],
\end{eqnarray*}
and 
\begin{eqnarray*}\label{mn.4}
&e=[e_1,d_2]=[e_1,f_2 \ e_2 \ f_2\ c\ a_2\ b_2\ a_2\ c\ f_2 \ e_2 \ f_2]
\xlongequal{[e_1,f_2]=e, \langle c, a_2\ b_2\ a_2 \rangle=e}\\
&[e_1, e_2 \ f_2\ a_2\ b_2\ a_2\ c \ a_2\ b_2\ a_2\ f_2 \ e_2]
\xlongequal{[a_2,f_2]=[a_2,e_2]=[b_2,f_2]=[b_2,e_2]=e}\\
&[e_1, e_2 \ f_2\ c \ f_2 \ e_2]
\xlongequal{[c,f_2]=e}[e_1,c\ e_2\ c].
\end{eqnarray*}
And we get the required relations of type (\ref{forkrel}).  

Another local piece from Figure \ref{Figmn} is given in Figure \ref{fork2}, and by similar computations, we get $[a_2, b_1~d~b_1] = e$ and $[f_1, d~f_2~d] = e$, which are of type (\ref{forkrel}).  
\begin{figure}[H]
\begin{center}
\scalebox{0.7}{\includegraphics{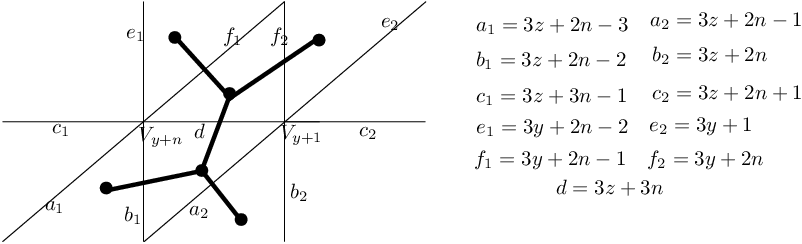}}
\end{center}
\setlength{\abovecaptionskip}{-0.15cm}
\caption{Another local piece from Figure~\ref{Figmn}}\label{fork2}
\end{figure}

As a result of the preceding calculations, the presentation of $G^{m,n}_1$ includes the additional $E_6$ basic cycle relations. Thus, $G^{m,n}_1\cong C_Y(T)/<$ all $E_6$ basic cycle relations$>$.

From the relations computed for group ${G^{m,n}_1}$, it is evident that it contains all the $E_6$ basic cycles, but none of the cycles of $m$ types, as follows (see Figures \ref{Figmn} and \ref{CPmncombine}):
\begin{itemize}\label{88}
    \item The cycle $(H^{m,n}_1)$ generated by generators $1, 2,..., 2n$;
    \item The cycle $(H^{m,n}_2)$ generated by generators $3n+1, 3n+2,..., 5n$;
    \item $\cdots\cdots\cdots$
    \item The cycle $(H^{m,n}_{m-1})$ generated by generators $3z+1, 3z+2,..., 3z+2n$;
    \item The cycle $(H^{m,n}_m)$ generated by generators $3y+1, 3y+2,..., 3y+2n$.
\end{itemize}

Also, here the following exact sequence holds:\\
$$
\begin{tikzcd}
0 \arrow[r] 
  & \pi_1(X^{m,n}_{gal}) \arrow[r] \arrow[d, ""] 
  & G^{m,n} \arrow[r] \arrow[d] 
  & S_{2mn} \arrow[r] \arrow[d, equal] 
  & 0 \\
0 \arrow[r] 
  & K_1(T) \arrow[r] 
  & G^{m,n}_1 \arrow[r] 
  & S_{2mn} \arrow[r] 
  & 0
\end{tikzcd}
$$

Proposition 4.1 in \cite{RTV} states that the kernel $K_1(T)$ in the second exact sequence is generated as a subgroup by the elements $\gamma_j \gamma^{-1}_i$. For instance, let us consider the cycle $H^{m,n}_1$ generated by generators $1, 2, \ldots, 2n$ in the dual graph. Define $$\gamma_i = (i+2)  \cdots  2n \cdot 1  \cdots  i,  \ \ for~ i=1, \ldots, 2n-2,$$ 
and set
$$\gamma_{2n-1}= 1 \cdot 2\cdots(2n-1),  \ \gamma_{2n}=2\cdot3\cdots(2n).$$
Then  $\gamma_{j}\gamma^{-1}_{1}\in K_1(T)$ for $~j=2,\ldots,2n$.
Hence, we obtain $2n-1$ distinct elements from cycle $H^{m,n}_1$ in $K_1(T)$. And that happens for each of cycles $H^{m,n}_i$, \ $i=1, \ldots, m$.

Via the lift map from $K_1(T)$ to $\pi_1(X^{m,n}_{gal})$, it follows directly that 
$\pi_1(X^{m,n}_{gal})$ contains the elements:
$$\gamma_{j}\gamma^{-1}_{1},~j=2,\ldots,2n.$$
   
Since cycle $H^{m,n}_1$ yields $(2n-1)$ elements ($\gamma_j\gamma^{-1}_1$ for $j = 2, \ldots, 2n$), we apply a similar method to show that, via the lift map, $\pi_1(X^{m,n}_{gal})$ contains $(2n-1)$ such elements for each cycle $H^{m,n}_i$ with $i = 1, 2, \ldots, m$.
In total, this gives $m(2n-1)$ elements.

\bigskip

Now, we prove the above elements are of infinite order and commute with each other.

The proof is divided into two steps. In the first step, we prove that the elements are torsion-free. As established previously, the group $C_Y(T)$ does not contain any of the $E_6$ basic cycle relations that are present in $G^{m,n}_1$.
According to Theorem~6.1 in \cite{RTV}, we have the following short exact sequence:
$$
0 
\rightarrow K(T)
\rightarrow C_Y(T)
\rightarrow S_{2mn}
\rightarrow 0,
$$
where $K(T)$ is isomorphic to $A_{t,n}$, with $t$ denoting the number of cycles in  dual graph $T$. The generators of $K(T)$ are elements of the form $\gamma_j\gamma^{-1}_i$, within each cycle.

Now consider group $G^{m,n}_1$. According to Lemma~\ref{ii}, for any $E_6$ basic cycle, if one relation of the form $\gamma_{j_0}\gamma^{-1}_{i_0}$ is imposed, then all other relations of the form $\gamma_j\gamma^{-1}_i$ within the same cycle are automatically satisfied. Consequently, $G^{m,n}_1$ is isomorphic to the quotient of $C_Y(T)$ by the normal subgroup generated by the elements $\gamma_j\gamma^{-1}_i$ from every $E_6$ basic cycle. Denoting this normal subgroup by $B$, we obtain the following short exact sequence:
$$
0 
\rightarrow K(T)/(B \cap K(T))
\rightarrow G^{m,n}_1
\rightarrow S_{2mn}
\rightarrow 0.
$$

The kernel $K(T)/(B \cap K(T))$ is generated by all elements $\gamma_j\gamma^{-1}_i$ corresponding to the cycle $(H^{m,n}_i), \ i=1, 2, \ldots, m-1, m$ (the elements arising from the $E_6$ basic cycle relations  have been quotiented out).

Suppose that there exists a torsion element in one of these generators. 
Consider the preimage of the above torsion element
in $K(T)$.
We denote it by $h$. This would imply that a power of this generator can be written as a product of other generators corresponding to the preimage of $\gamma_j\gamma^{-1}_i \ (i=1, 2, \ldots, m-1, m) $ from the $E_6$ basic cycle.
Using \cite[Theorem~6.1]{RTV}, we can write such relations as follows:
$$ (h)^n=\sum_{i,j} {a_{ij}}\gamma_j\gamma^{-1}_i.~~~~~(*)$$

However, since $K(T) \cong A_{t,n}$, and by the definition of $A_{t,n}$, its generators satisfy only commutation relations of the form $[x_{ij}, y_{kl}] = 1$ and no others, it follows that $A_{t,n}$ admits no relations of form~$(*)$. This leads to a contradiction. Therefore, the elements $\gamma_{i}\gamma_{1}^{-1} \in K(T)$ for $i = 2, \ldots, 2n$ (studied in Section~\ref{sec-mn}) are torsion-free.

In the second step, we prove that these elements commute pairwise. The commutativity of elements from different cycles $(H^{m,n}_i)$ can be verified directly by computation. The key part is to prove that within the same cycle $(H^{m,n}_i)$, elements of form $\gamma_i\gamma^{-1}_1$ and $\gamma_j\gamma^{-1}_1$ commute. To show that $\gamma_i\gamma^{-1}_1\gamma_j\gamma^{-1}_1 = \gamma_j\gamma^{-1}_1\gamma_i\gamma^{-1}_1$, it suffices to prove the equivalent relation:
$
\gamma_i\gamma^{-1}_1\gamma_j = \gamma_j\gamma^{-1}_1\gamma_i$,
where $\gamma_i=(i+2)\cdot (i+3)\cdots(2n-1)\cdot(2n)\cdot 1\cdot2 \cdots i$, here $i=1,2,\dots,2n-2 $, and $\gamma_{2n-1}=1\cdot 2\cdots(2n-2)\cdot(2n-1)$, $\gamma_{2n}=2\cdot 3\cdots(2n-1)\cdot(2n)$. 

Here we only prove that the case 
\begin{equation}\label{ff}
   \gamma_{2n-1}\gamma^{-1}_1\gamma_{2n} = \gamma_{2n}\gamma^{-1}_1\gamma_{2n-1} 
\end{equation}
holds for the cycle $(H^{m,n}_1)$ (the other cases have similar proofs). Since $i=i^{-1}$ for any generator $i$ in the group $G^{m,n}_1$, we have $\gamma^{-1}_1=1^{-1} \cdot (2n)^{-1}\cdot (2n-1)^{-1} \cdots 4^{-1}\cdot 3^{-1}=1 \cdot (2n)\cdot (2n-1) \cdots 4\cdot 3$. 

Now we expand both sides of relation (\ref{ff}) by substituting the above expressions. LHS (resp. RHS) is the left (resp. right) side of (\ref{ff}). We get the following equality: 
\begin{equation}\label{gg1}
\begin{aligned}
\text{LHS} &= 1\cdot 2\cdots(2n-2)\cdot(2n-1) \cdot 1 \cdot (2n)\cdot (2n-1) \cdots 4\cdot 3 \cdot 2\cdot 3\cdots(2n-1)\cdot(2n) \\
&= 2\cdot 3\cdots(2n-1)\cdot(2n) \cdot 1 \cdot (2n)\cdot (2n-1) \cdots 4\cdot 3 \cdot 1\cdot 2\cdots(2n-2)\cdot(2n-1)=\text{RHS}.
\end{aligned}
\end{equation}
We simplify (\ref{gg1}) by using relations $[1,i]=e,~i=3,~4,~\dots,~(2n-1)$ and $\langle 1,~2\rangle=e=\langle 1,~2n\rangle$. We get the following equality:
\begin{equation}\label{gg2}
\begin{aligned}
\text{LHS} &= 2\cdot 3\cdots(2n-2)\cdot(2n-1) \cdot (2n)\cdot (2n-1) \cdots 4\cdot 3 \cdot 2\cdot 3\cdots(2n-1)\cdot(2n)\\
&= 3\cdot 4\cdots(2n-2)\cdot(2n-1) \cdot (2n)\cdot (2n-1) \cdots 4\cdot 3 \cdot 2\cdot3\cdots(2n-2)\cdot(2n-1)=\text{RHS}.
\end{aligned}
\end{equation}
Then we use relations  $[2n,i]=e,~i=2,~3,~\dots,~(2n-2)$ and $\langle 2n-1,~2n\rangle=e$, and get
\begin{equation}\label{gg3}
\begin{aligned}
\text{LHS} &= 2\cdot 3\cdots(2n-2)\cdot(2n-1) \cdot (2n-2)\cdot (2n)\cdot(2n-3) \cdots 4\cdot 3 \cdot 2\cdot 3\cdots(2n-1)\cdot(2n)\\
&= 3\cdot 4\cdots(2n-2)\cdot(2n-1) \cdot (2n-2)\cdot(2n) \cdot (2n-3) \cdots 4\cdot 3 \cdot 2\cdot3\cdots(2n-2)\cdot(2n-1)\\&=\text{RHS}.
\end{aligned}
\end{equation}
Now we use relations  $[2n-1,i]=e,~i=2,~3,~\dots,~(2n-3)$ and $\langle 2n-2,~2n-1\rangle=e$ to get
\begin{equation}\label{gg4}
\begin{aligned}
\text{LHS} &= 2\cdot 3\cdots(2n-3)\cdot(2n-2)\cdot(2n-3) \cdot (2n-1)\cdot (2n)\cdot(2n-4) \cdots 4\cdot 3 \cdot 2\cdot 3\cdots(2n-1)\\&\cdot(2n) 
= 3\cdot 4\cdots(2n-3)\cdot(2n-2)\cdot(2n-3) \cdot (2n-1)\cdot(2n) \cdot (2n-4) \cdots 4\cdot 3 \cdot 2\cdot3\\&\cdots(2n-2)\cdot(2n-1)=\text{RHS}.
\end{aligned}
\end{equation}
The simplification proceeds by using $[2n-2,i] = e$, $i = 2, 3, \dots, (2n-4)$, and $\langle 2n-3, ~2n-2\rangle = e$. Ultimately, we obtain the relation: $$\langle 2n, 2n-1\rangle = e.$$
It is clear that this relation already exists, hence the equivalent relation (\ref{ff}) holds. A similar argument applies to any other pair of commuting elements.

Therefore, the $m(2n-1)$ elements generate a subgroup isomorphic to $\mathbb{Z}^{m(2n-1)}$, which completes the proof of the theorem.
\end{proof}

\subsection{Open questions}

Computing the fundamental group $\pi_1(X^{m,n}_{gal})$ for a general surface $X^{m,n}$ is considerably more complex than for the special cases previously studied. We conclude by posing the following interesting open questions:

 \begin{Qst}Is $\pi_1(X^{m,n}_{gal})$ an abelian group?
\end{Qst}

\begin{Qst}Can one explicitly express $\pi_1(X^{m,n}_{gal})$ as a direct or semidirect product of simpler groups?
\end{Qst}

Answers to these questions will significantly deepen our understanding of the moduli space of algebraic surfaces.

\section{Estimation of the irregularity of the Galois covers}\label{bgzd}
In this section, we estimate the irregularity of the Galois covers using the method of Liedtke~\cite{Li}.

Let $f: X \longrightarrow \mathbb{CP}^2$
be a generic projection of degree $l$, with branch curve
$S \subset \mathbb{CP}^2.$
Consider an affine chart $\mathbb{CA}^2$ of $\mathbb{CP}^2$, and let $
D := S \cap \mathbb{CA}^2$
be the affine portion of the branch curve.

From Lemma 3.2 of \cite{L}
there is a short exact sequence
$$
1 \to \pi_1(X^{\mathrm{aff}}_{\mathrm{gal}}) \to \pi_1(\mathbb{CA}^2 - D, x_0)/\langle \Gamma_i^2 \rangle\ \xrightarrow{\psi} S_l \to 1.
$$
Here ${\Gamma_i}$ are the generators of $\pi_1(\mathbb{CA}^2-D, x_0)$, and 
 $X^{\mathrm{aff}}_{\mathrm{gal}}$ denotes an affine chart of the Galois cover $X_{\mathrm{gal}}$ for a general projection $f$.

We define now $\overline {C_{proj}}$.
\begin{Def}
Let $C_{\text{aff}}$ be the subgroup normally generated by the following elements inside $\pi_1(\mathbb{CA}^2-D, x_0)$:
\begin{itemize}
\item $[\gamma\Gamma_i\gamma, \Gamma^{-1}_j]$~~ if $\psi(\gamma\Gamma_j\gamma^{-1})$ and $\psi(\Gamma_i)$ are disjoint transpositions,
\item $<\gamma\Gamma_i\gamma, {\Gamma^{-1}_{j}}> $ if $\psi(\gamma\Gamma_i\gamma^{-1})$ and $\psi(\Gamma_j)$ have precisely one letter in common,
\end{itemize}
where $\gamma$ runs through $\pi_1(\mathbb{CA}^2-D, x_0)$.

We define the image of $C_{\text{aff}}$ in $\pi_1(X_{gal})$ by $C_{proj}$. Let $\overline {C_{proj}}$ be the image of $C_{proj}$ in the abelianization of $\pi_1(X_{gal})$.
\end{Def}

Our main tool is the following result (see Corollary 4.9 of \cite{L}):

\begin{cor}\label{L} 
The rank of $H_1(X_{gal}, \mathbb{Z})/\overline{C_{proj}}$ as an abelian group is equal to $(l - 1)$ times the rank of $H_1(X, \mathbb{Z})$,  where $H_1(X_{gal}, \mathbb{Z})$ and $H_1(X, \mathbb{Z})$ are the first homology groups of $X_{gal}$ and $X$, respectively, and $\overline{C_{proj}}$ is a subgroup  of $H_1(X_{gal}, \mathbb{Z})$.
\end{cor}

In the context of this paper, it can be calculated that the degree of the general cover is $2mn$. Since $X^{m,n}$ is an elliptic surface, the rank of $H_1(X^{m,n}, \mathbb{Z})$ is $2$. Therefore, by Corollary \ref{L}, the rank of $H_1(X^{m,n}_{gal}, \mathbb{Z}) / \overline{C_{proj}}$ is $(2mn - 1) \times 2 = 4mn - 2$. Consequently, the rank of $H_1(X^{m,n}_{gal}, \mathbb{Z})$ itself is at least $4mn - 2$. According to the Hodge decomposition (see \cite{GH}), we have $\mathrm{rank}~(H_1(X^{m,n}_{gal}, \mathbb{Z})) = 2q(X^{m,n}_{gal})$, where $q(X^{m,n}_{gal})$ is the regularity of the Galois cover, and this implies that the irregularity satisfies $q(X^{m,n}_{gal}) \geq 2mn - 1$.

This lower bound is in fact sharp. To see this, consider the surfaces $X^{1,n}$ studied in \cite{AGTV02}. These surfaces are constructed as Galois covers of generic projections from the product surface $E \times \mathbb{P}^1$, where $E$ is an elliptic curve.

Let $F$ be a very ample line bundle of degree $n$ on $E$. The line bundle $L := F \boxtimes O_{\mathbb{P}^1}(d)$ on $X^{1,n}$ is very ample. If we choose $n$ sufficiently large and $d=1$, then $L$ gives rise to a generic projection $f : X^{1,n} \rightarrow \mathbb{P}^2$. The degree $n = \mathrm{deg} f$ equals the self-intersection of $L$, which is equal to $2n$. For this surface, it is shown that $\pi_1(X^{1,n}_{gal}) \cong \mathbb{Z}^{4n - 2}$. Since the first homology group is the abelianization of the fundamental group, it follows that $H_1(X^{1,n}_{gal}, \mathbb{Z}) \cong \mathbb{Z}^{4n - 2}$, which indicates that $q(X^{1,n}_{gal}) = \frac{4n - 2}{2} = 2n - 1$.

\section{Chern number and index}\label{Chern}

We study the surface $X^{m,n}$, which has a degeneration shown in Figure~\ref{Figmn}. In this section, we compute the Chern numbers of the Galois cover of this surface using the following theorem:

\begin{thm}\cite[Proposition\, 0.2]{MoTe87} \label{chern}
Let $S$ be the branch curve of an algebraic surface $X$. Denote the
degree of the generic projection by $b$, and $f: X \rightarrow \mathbb{CP}^2$ is the generic projection.

The Chern numbers of the Galois covers of the surfaces are given in the following formulas:
\begin{enumerate}
\item  $c_1^2(X_{gal})=\frac{b!}{4} (h-6)^2,$
\item $c_2(X_{gal})=b! (\frac{h^2}{2}-\frac{3h}{2}+3-\frac{3d}{4}-\frac{4\rho}{3}),$

where $h=\deg S$, $d=$ number of nodes in $S$,
$\rho=$ number of cusps in $S$.
\end{enumerate}
\end{thm}

Figure~\ref{Figmn}, above, shows the degeneration of the surfaces $X^{m,n}$. From this configuration, we observe that all relations from nodes and cusps originate from vertices of two types~(intersections of either two or six lines), and relations of type \ref{par}.

The number of vertices where two lines intersect is $2n$. Each of the vertices $V_2, \dots, V_n$ and $V_{n+1}, V_{n+2}, \dots, V_{2n}$ contributes three cusps. Therefore, the total number of cusps arising from these vertices is $6n$.

The number of vertices where six lines intersect is $n(m-1)$. According to \cite{ATXY}, each such vertex contributes 24 nodes and 24 cusps. Therefore, the total number of nodes (respectively, cusps) arising from these vertices is $24 n (m-1)$.

Relations of type~\ref{par}, introduced in the preliminaries, also contribute commutation relations. According to Figure \ref{Figmn}, the number of such commutation relations is given by the formulas

\begin{itemize}
   \item $4\cdot(C^2_{3mn-n}-2n-15(m-1)n+(m-1))$, if $n=2$,
   \item $4\cdot(C^2_{3mn-n}-2n-15(m-1)n)$, if $n\geq3$,
\end{itemize}
where $C^r_n=\frac{n!}{r! \ (n-r)!}$.

In conclusion, we have $b=2mn, \ h=6mn-2n, \ \rho=6n+24~n\cdot(m-1)=24mn-18n$, and the following formulas:

\begin{itemize}
    \item when $n=2$, we have
    
    $d=24~n\cdot(m-1)+4\cdot(C^2_{3mn-n}-2n-15(m-1)n+(m-1))=72m^2-128m+64,$
     \item and when $n\geq3$, we have 
     
       $d=24~n\cdot(m-1)+4\cdot(C^2_{3mn-n}-2n-15(m-1)n)=18m^2 n^2-12mn^2-42mn+2n^2+30n.$
\end{itemize}

Furthermore, in the degeneration of $X^{m,n}$, we have $\deg S = 6mn - 2n$.

The computation of Chern numbers is a crucial step in studying the moduli spaces of algebraic surfaces， see \cite{Gie}.
By the formula in Theorem \ref{chern}, we get the following Chern numbers: 

\begin{itemize}
    \item when $n=2$, we have 
    
     $c_1^2(X^{m,n}_{gal})=(4m)!~(36m^2-60m+25)$
     
     $c_2(X^{m,n}_{gal})=(4m)!~(18m^2-34m+17),$
     \item and when $n\geq3$, we have
     
  $c_1^2(X^{m,n}_{gal})=(2mn)!~(9m^2 \cdot n^2-6mn^2-18mn+n^2+6n+9)$
  
   $c_2(X^{m,n}_{gal})=(2mn)!~(9/2\cdot m^2  n^2-3mn^2-19/2\cdot mn+1/2\cdot n^2+9/2\cdot n+3).$
\end{itemize}

 In the geography problem of surfaces, the index $\tau(X_{gal}) = \frac{1}{3} \left(c_1^2(X_{gal}) - 2c_2(X_{gal})\right)$ can concisely reflect the topological properties of the surface. Moreover, the index  of surfaces, in conjunction with the Seiberg–Witten theory and the Dirac operator index, constrains the existence of metrics with positive scalar curvature (see \cite{BHPV}). Based on the results regarding the above calculations of the Chern numbers, we have:

\begin{itemize}
   \item $\tau(X^{m,n}_{gal})=(4m)!/3~\cdot(8m-9)$, \ if $n=2$.
   \item $\tau(X^{m,n}_{gal})=(2mn)!/3~\cdot (mn-3n+3)$, \ if $n\geq3$.
\end{itemize}

Then we get Theorem \ref{s}, which provides an exact expression for the $\tau$-invariant of $X^{m,n}_{gal}$ for different parameters, along with a detailed analysis of the sign of its index. The result unifies both the cases $n=2$ and $n \geq 3$, and reveals the significant influence of the parameters $m$ and $n$ on the geometric properties of the surface: as the parameters vary, the index of the surface is shown to transition from negative through zero to positive values, reflecting a profound relationship between its geometric structure and topological complexity. The proof of the theorem thus provides a pivotal computational foundation and theoretical reference for further study of the invariants and classification of such algebraic varieties.

\paragraph{Acknowledgments:} We express our gratitude to Prof. Xianhua Li for discussing the construction of the subgroups. 
This research was partly supported by NSFC~(No.12331001, 12271384), Natural Science Research Project of Anhui Educational Committee (No.2025AHGXZK40337), Anhui Science and Technology University 2025 (Stable) Talent Recruitment Program~(No.2025qhxm28), DFG Grant~(No.537912683), and  the Key Discipline Construction Project of Anhui Science and Technology University~(No.XK-XJGY002).

\end{document}